\documentclass[a4paper,reqno,10pt]{amsart}

\usepackage[a4paper,scale={0.72},marginratio={1:1}]{geometry}

\usepackage{amsmath}
\usepackage{dsfont}
\usepackage{perpage}
\usepackage{color}

\usepackage{hyperref}

\usepackage{microtype}

\numberwithin{equation}{section}

\newtheorem{theorem}{Theorem}[section]
\newtheorem{proposition}[theorem]{Proposition}

\newtheorem{remark}[theorem]{Remark}
\newtheorem{definition}[theorem]{Definition}
\newtheorem{assumption}[theorem]{Assumption}

\newcommand{\N}{\mathbb{N}}
\newcommand{\Z}{\mathbb{Z}}
\newcommand{\R}{\mathbb{R}}

\newcommand{\dd}{\mathrm{d}}

\renewcommand{\P}{\mathrm{P}}
\newcommand{\Q}{\mathrm{Q}}
\newcommand{\bP}{\mathbf{P}}
\newcommand{\bQ}{\mathbf{Q}}
\newcommand{\E}{\mathrm{E}}

\newcommand{\ind}{{\mathds 1}}

\newcommand{\cR}{\mathcal{R}}

\newcommand{\bV}{\boldsymbol{V}}
\newcommand{\bz}{\boldsymbol{z}}

\newcommand{\sfP}{\mathsf{P}}
\newcommand{\sfE}{\mathsf{E}}
\newcommand{\sfC}{\mathsf{C}}

\newcommand\rig{\mathrm{fin}}
\newcommand\bulk{\mathrm{bulk}}
\newcommand\bridge{\mathrm{bri}}
\newcommand\meander{\mathrm{mea}}
\newcommand\excursion{\mathrm{exc}}
\newcommand\meandertwo{\mathrm{mea2}}
\newcommand\excursiontwo{\mathrm{exc2}}

\MakePerPage[2]{footnote} 

\date{\today}

\title[Maximum of conditioned random walks]{On the maximum
of conditioned random walks\\ and tightness
for pinning models}

\author{Francesco Caravenna}
\address{Dipartimento di Matematica e Applicazioni, 
Universit\`a degli Studi di Milano-Bicocca,
via Cozzi 55,
20125 Milano, Italy}

\email{francesco.caravenna@unimib.it}



\begin{document}

\begin{abstract}
We consider real random walks with finite variance.
We prove an optimal integrability result for the diffusively rescaled maximum,
when the  walk or its bridge
is conditioned to stay positive,
or to avoid zero.
As an application, 
we prove tightness under diffusive rescaling for general
pinning and wetting models based on random walks.
\end{abstract}

\maketitle

\section{Introduction}

In this paper we deal with random walks on $\R$,
with zero mean and finite variance.

\smallskip

In Section~\ref{sec:rw} we consider the random walks, or their bridges,
conditioned to stay positive on a finite time interval. We prove that the maximum
of the walk, diffusively rescaled, has a \emph{uniformly integrable square}.
The same result is proved under the conditioning that the walk avoids zero.

\smallskip

In Section~\ref{sec:meas} we present 
an application to \emph{pinning and wetting models}
built over random walks.
More generally, we consider probabilities 
which admit suitable
regeneration epochs,
which cut the path into independent ``excursions''.
We prove that these models, under diffusive rescaling,
are tight in the space of continuous functions.
This fills a gap in
the proof of \cite[Lemma~4]{cf:DGZ}.

\smallskip

Sections~\ref{sec:prmain2}, \ref{sec:m3}, \ref{sec:main1} contain the proofs.

\smallskip

This paper generalizes and supersedes 
the unpublished manuscript 
\cite{cf:CGZ?}.

\section{Random walks conditioned to stay positive, or to avoid zero}
\label{sec:rw}

We use the conventions $\N := \{1,2,3,\ldots\}$ and $\N_0 := \N \cup \{0\}$.
Let $(X_i)_{i\in\N}$ be i.i.d.\ real random variables.
Let $(S_n)_{n\in\N_0}$ be the associated random walk:
\begin{equation*}
	S_0 := 0 \,, \qquad
	S_n := X_1 + \ldots + X_n \quad \text{for} \ n \in \N \,.
\end{equation*}

\begin{assumption} \label{ass:rw}
$\E[X_1]=0$, $\E[X_1^2]= \sigma^2 < \infty$,
and one of the following cases hold.
\begin{itemize}
\item \emph{Discrete case}.
The law of $X_1$ is integer valued and, for simplicity, 
the random walk is aperiodic, i.e.\ $\P(S_n = 0) > 0$ for large $n$,
say $n \ge n_0$.

\smallskip
\item \emph{Continuous case}.
The law of $X_1$ has a density with respect to the Lebesgue measure, and
the density of $S_n$ is essentially bounded for some $n \in \N$:
\begin{equation*}
	f_n(x) \,:=\,  \frac{\P(S_n \in \dd x)}{\dd x} \,\in\, L^\infty \,.
\end{equation*}
It follows that for large $n$,
say $n \ge n_0$, $f_n$ is bounded and continuous, and $f_n(0) > 0$.
\end{itemize}
\end{assumption}

Let us denote by $\P_n$ the law of the first $n$ steps of the walk:
\begin{equation} \label{eq:Pn}
	\P_n := \P\big( (S_0, S_1, \ldots, S_n) \in \,\cdot\, \big) \,.
\end{equation}
Next we define the laws of the
\emph{meander}, \emph{bridge} and \emph{excursion}:
\begin{equation} \label{eq:laws}
\begin{split}
	\P_n^\meander(\,\cdot\,) & := \P\big((S_0, S_1, \ldots, S_n) \in \,\cdot\, 
	\,\big|\, S_1 > 0, S_2 > 0,
	\ldots, S_n>0\big) \,, \\
	\P_n^\bridge(\,\cdot\,) & := 
	\P\big((S_0, S_1, \ldots, S_n) \in \,\cdot\,\, \big| \, S_n=0 \big) \,, \\
	\P_n^\excursion(\,\cdot\,) & := \P\big((S_0, S_1, \ldots, S_n) \in \,\cdot\, 
	\,\big|\, S_1 > 0, S_2 > 0,
	\ldots, S_{n-1}>0, S_n = 0\big) \,.
\end{split}
\end{equation}
In Remark~\ref{rem:technical} below
we discuss the conditioning on $\{S_n = 0\}$, and periodicity issues.

\smallskip

Our main result concerns the integrability of the absolute maximum of the walk:
\begin{equation} \label{eq:Mn}
	M_n := \max_{0 \le i \le n} |S_i| \,.
\end{equation}

\begin{theorem}\label{th:main2}
Let Assumption~\ref{ass:rw} hold. Then
$M_n^2 / n$ is uniformly integrable under any of the laws $\Q \, \in \, \big\{ \P_n,  \, \P_n^\bridge, \,
\P_n^\meander, \, \P_n^\excursion \big\}$:
\begin{equation}\label{eq:ui}
	\lim_{K \to \infty} \ \sup_{n\in\N} \
	\E_\Q \bigg[ \frac{M_n^2}{n} \, \ind_{\big\{
	\frac{M_n^2}{n} > K \big\}} \bigg] = 0 \,.
\end{equation}
\end{theorem}

\smallskip

The proof of Theorem~\ref{th:main2}, given in Section~\ref{sec:prmain2},
comes in three steps. First we exploit \emph{local limit
theorems}, to remove the conditioning on $\{S_n=0\}$ and just
deal with $\P_n$, $\P_n^\meander$.
Then we use \emph{martingale arguments}, to get rid of the maximum
$M_n$ and focus on $S_n$.
Finally we use \emph{fluctuation theory},
to perform sharp computations on the law of $S_n$.

\begin{remark}
For a symmetric random walk,
the bound $M_{n}^2 \ge X_{n}^2 \, 
\ind_{\{S_{n-1} \ge 0, \, X_{n} \ge 0\}}$
gives
\begin{equation} \label{eq:eabo}
	\E \bigg[ \frac{M_{n}^2}{n} \, \ind_{\big\{
	\frac{M_{n}^2}{n} > K \big\}} \bigg]
	\,\ge\, \frac{1}{4} \, 
	\E\bigg[ \frac{X_{1}^2}{n} \, 
	\ind_{\big\{\frac{X_{1}^2}{n} > K \big\}} \bigg] \,.
\end{equation}
Given $n\in\N$, we can choose the law of $X_1$
so that the right hand side
vanishes as slow as we wish,  as $K \to \infty$.
Thus \eqref{eq:ui} cannot be improved, without further assumptions.
\end{remark}

\smallskip

We next introduce the laws of the random walk and bridge \emph{conditioned
to avoid zero}:
\begin{equation} \label{eq:laws2}
\begin{split}
	\P_n^\meandertwo(\,\cdot\,) & := \P\big((S_0, S_1, \ldots, S_n) \in \,\cdot\, 
	\,\big|\, S_1 \ne 0, S_2 \ne 0,
	\ldots, S_n \ne 0\big) \,, \\
	\P_n^\excursiontwo(\,\cdot\,) & := \P\big((S_0, S_1, \ldots, S_n) \in \,\cdot\, 
	\,\big|\, S_1 \ne 0, S_2 \ne 0,
	\ldots, S_{n-1} \ne 0, S_n = 0\big) \,.
\end{split}
\end{equation}
In the continuous case $\P(S_n \ne 0) = 1$, so we have trivially
$\P_n^\meandertwo = \P_n$
and $\P_n^\excursiontwo = \P_n^\bridge$.
In the discrete case, however, 
the conditioning on $\{S_n \ne 0\}$ has a substantial effect:
$\P_n^\meandertwo$ and $\P_n^\excursiontwo$
are close to ``two-sided versions'' of 
$\P_n^\meander$ and $\P_n^\excursion$
(see \cite{cf:Bel,cf:K}).

\smallskip

We prove the following analogue of Theorem~\ref{th:main2}.

\begin{theorem}\label{th:main3}
Let Assumption~\ref{ass:rw} hold.
Then $M_n^2/n$ under $\P_n^\excursiontwo$
or $\P_n^\meandertwo$ is uniformly integrable.
\end{theorem}

\smallskip

Theorem~\ref{th:main3} is proved in Section~\ref{sec:m3}.
We first use local limit theorems to
reduce the analysis to $\P_n^\meandertwo$, 
as for Theorem~\ref{th:main2}, but
we can no longer apply martingale techniques. 
We then exploit direct path arguments to \emph{deduce}
Theorem~\ref{th:main3} from Theorem~\ref{th:main2}.

\begin{remark}\label{rem:technical}
The laws $\P_n^\bridge$, $\P_n^\excursion$, $\P_n^\excursiontwo$ are well-defined
for $n \ge n_0$ --- since
$\P(S_n = 0) > 0$ or $f_n(0) > 0$,  see Assumption~\ref{ass:rw} ---
but not obviously for $n < n_0$. This is quite immaterial for our goals, since
uniform integrability is essentially an asymptotic property:
we can take any definition for these laws for $n < n_0$,
as long as we have $M_n \in L^2$.

We also stress that we require aperiodicity in Assumption~\ref{ass:rw}
only for notational convenience. If a discrete random walk has period $T \ge 2$, then
Theorems~\ref{th:main2} and~\ref{th:main3} still hold,
with essentially no change in the proofs,
but for the the laws $\P_n^\bridge$, $\P_n^\excursion$, $\P_n^\excursiontwo$
to be well-defined
we have to restrict $n \in T \N$, to ensure that $\P(S_n = 0) > 0$ for large $n$.
\end{remark}

\section{Tightness for pinning and wetting models}
\label{sec:meas}

We prove tightness under diffusive rescaling
for \emph{pinning and wetting models}, see Subsection~\ref{sec:pinning},
exploiting the property that these models 
have independent excursions\footnote{In this section
the word ``excursion'' has a more general meaning than in
Section~\ref{sec:rw}.}, conditionally on their
zero level set. It is simpler and more transparent to 
work with general probabilities which enjoy (a generalization of) this property,
that we now define.

\subsection{A sharp criterion for tightness based on excursions}

Given $t\in\N$, we use the
shorthands
\begin{equation*}
	[t] := \{0,1,\ldots,t\} \,, \qquad
	\R^{[t]} = \{x = (x_0, x_1, \ldots, x_t): \ x_i \in \R\}
	\simeq \R^{t+1} \,.
\end{equation*}
We consider probabilities $\bP_N$
on paths $x = (x_0, \ldots, x_N) \in \R^{[N]}$
which admit \emph{regeneration epochs in their zero level set}.
To define $\bP_N$, we need three ingredients:
\begin{itemize}
	\item the \emph{regeneration law} $p_N$ is a
	probability on the space of
	subsets of $[N]$ which contain $0$;

	\item the \emph{bulk excursion laws} $P_t^\bulk$, $t\in\N$, are probabilities
	on $\R^{[t]}$ with
	$P_t^\bulk(x_0 = x_t = 0) = 1$;

	\item the \emph{final excursion laws}
	$P_t^\rig$, $t\in\N$, are probabilities on $\R^{[t]}$ with
	$P_t^\rig(x_0 = 0) = 1$.

\end{itemize}

\begin{definition}\label{def:bPN}
The law
$\bP_N$ is the probability on $\R^{[N]}$ under which
the path $x = (x_0, x_1, \ldots, x_N)$ is built as follows.
\begin{enumerate}
\item First sample the number $n$
and the locations $0 =: t_1 < \ldots < t_n \le N$ of the \emph{regeneration epochs},
with probabilities
$p_N(\{t_1,\ldots, t_n\})$.

\item Then write the path $x$ as a concatenation of $n$ \emph{excursions}
$x^{(i)}$, with $i=1,\ldots, n$:
\begin{equation*}
	x^{(i)} := (x_{t_i}, \ldots, x_{t_{i+1}}) \,, \qquad
	\text{with \ $t_{n+1} := N$} \,.
\end{equation*}

\item Finally, given the regeneration epochs,
sample the excursions $x^{(i)}$ independently, with marginal laws
$P^\bulk_{t_{i+1}-t_i}$ for $i = 1, \ldots, n-1$
and (in case $t_n < N$) $P^\rig_{N-t_n}$ for $i=n$.
\end{enumerate}
\end{definition}

Let $C([0,1])$ be the space of continuous functions
$f:[0,1] \to \R$, with the topology of uniform convergence.
We define the
\emph{diffusive rescaling} operator
$\cR_N: \R^{[N]} \to C([0,1])$
\begin{equation*}
\begin{split}
	\cR_N(x) := & \ \Big\{ \text{linear interpolation of } 
	\ \tfrac{1}{\sqrt{N}} x_{Nt} \
	\text{ for } t \in \big\{0, \tfrac{1}{N}, \ldots, \tfrac{N-1}{N}, 1 \big\}
	\Big\}
\end{split}
\end{equation*}
We give optimal conditions under which the laws $\bP_N \circ \cR_N^{-1}$, called
\emph{diffusive rescalings} of $\bP_N$, are tight. Remarkably, we make \emph{no
assumption on the regeneration laws $p_N$}.

\smallskip

\begin{theorem}\label{th:main1}
Let $\bP_N$ be
as in Definition~\ref{def:bPN}.
The diffusive rescalings
$(\bP_N \circ \cR_N^{-1})_{N\in\N}$ are tight in $C([0,1])$,
for an arbitrary choice of the regeneration laws $(p_N)_{N\in\N}$,
if and only if the following conditions hold:
\begin{enumerate}
\item \label{it:1} the diffusive rescalings
$(P_t^\bulk \circ \cR_t^{-1})_{t\in\N}$ and $(P_t^\rig \circ \cR_t^{-1})_{t\in\N}$
are tight in $C([0,1])$;

\item \label{it:2} the bulk excursion law satisfies the following
integrability bound:
\begin{equation}\label{eq:keycond}
	\sup_{t\in\N} \, P_t^\bulk\bigg( \frac{\max_{0 \le i \le t} |x_i|}{\sqrt{t}}
	> a \bigg) = o\bigg(\frac{1}{a^2}\bigg) \qquad
	\text{as } a \uparrow \infty \,.
\end{equation}
\end{enumerate}
\end{theorem}

\noindent
We point out that a slightly weaker version of Theorem~\ref{th:main1} was proved
in \cite{cf:CGZ?}.

\smallskip

To make a link with the previous section, we set
$M_t := \max_{0 \le i \le t} |x_i|$ and observe that
\begin{equation*}
	P_t^\bulk\bigg( \frac{\max_{0 \le i \le t} |x_i|}{\sqrt{t}}
	> a \bigg) \le 
	\frac{1}{a^2} \, E_t^\bulk \bigg[ \displaystyle \frac{M_t^2}{t}
	\, \ind_{\big\{\frac{M_t^2}{t} > a\big\}} \bigg] \,.
\end{equation*}
Thus condition~\eqref{it:2} in Theorem~\ref{th:main1} is satisfied
\emph{if $M_t^2/t$ is uniformly integrable under $P_t^\bulk$}.
We then obtain the following corollary of Theorems~\ref{th:main2} and~\ref{th:main3}.

\begin{proposition}\label{th:cor}
Condition~\eqref{it:2} in Theorem~\ref{th:main1} is satisfied
if $P_t^\bulk$ is chosen among $\{\P_t^\bridge, \P_t^\excursion,
\P_t^\excursiontwo\}$,
see \eqref{eq:laws} and \eqref{eq:laws2},
for a random walk satisfying Assumption~\ref{ass:rw}.
\end{proposition}

\begin{remark}\label{rem:n0}
Condition~\eqref{it:1} in Theorem~\ref{th:main1} is satisfied too,
if $P_t^\bulk$ is chosen among $\{\P_t^\bridge, \P_t^\excursion, \P_t^\excursiontwo\}$
and $P_t^\rig$
is chosen among $\{\P_n, \P_n^\meander, \P_n^\meandertwo \}$,
under Assumption~\ref{ass:rw}.
Indeed, the diffusive rescalings of
$\P_n$, $\P_n^\bridge$, $\P_n^\meander$ and $\P_n^\excursion$
converge weakly to Brownian motion \cite{cf:Donsker}, bridge
\cite{cf:Lig,cf:DGZ}, meander \cite{cf:B} and excursion \cite{cf:CarCha2};
in the discrete case,  the diffusive rescalings
of $\P_n^\meandertwo$ and $\P_n^\excursiontwo$ converge weakly
to two-sided Brownian meander \cite{cf:Bel} and excursion \cite{cf:K}.
\end{remark}

\smallskip
\subsection{Pinning and wetting models}
\label{sec:pinning}

An important class of laws $\bP_N$ 
to which Theorem~\ref{th:main1} applies is
given by pinning and wetting models (see \cite{cf:Gia,cf:Gia2,cf:dH} 
for background).

\smallskip

Fix a random walk $(S_n)_{n\in\N_0}$ as in
Assumption~\ref{ass:rw}
and a real sequence $\xi = (\xi_n)_{n\in\N}$
(\emph{environment}).
For $N\in\N$, the \emph{pinning model} $\bP_N^{\xi}$
is the law on $\R^{[N]}$ defined as follows.
\begin{itemize}
\item \emph{Discrete case}. We define
\begin{equation*}
	\frac{\bP_N^\xi\big( (S_0, \ldots, S_N) = (s_0, \ldots, s_N) \big)}
	{\P\big( (S_1, \ldots, S_N) = (s_1, \ldots, s_N) \big)} := 
	\frac{e^{\sum_{n=1}^N \xi_n \, \ind_{\{s_n=0\}}}}{Z_N^\xi} \,,
\end{equation*}
where $Z_N^\xi$ is a suitable normalizing constant, called \emph{partition function}.

\smallskip
\item \emph{Continuous case}.
We assume that $\xi_n \ge 0$ for all $n\in\N$ and we define $\bP_N^\xi$ by
\begin{equation*}
	\bP_N^\xi\big( (S_0, \ldots, S_n) \in (\dd s_0, \ldots, \dd s_n) \big) := 
	\delta_0(\dd s_0) \, \frac{\prod_{n=1}^N \big(
	f(s_n - s_{n-1}) \, \dd s_n \,+\, \xi_n \, \delta_0(\dd s_n) \big)}{Z_N^\xi} \,,
\end{equation*}
where $f(\cdot)$ is the density of $S_1$
and $\delta_0(\cdot)$ is the Dirac mass at $0$.
\end{itemize}
Note that $\bP_N^\xi$ fits Definition~\ref{def:bPN} with
regeneration epochs $\{k \in [N]: \ s_k = 0\}$
(the whole zero level set)
and $P_t^\bulk = \P_t^\excursiontwo$, $P_t^\rig = \P_t^\meandertwo$
(which means $P_t^\bulk = \P_t^\bridge$, $P_t^\rig = \P_t$ in the continuous case).

\smallskip

Another example of law $\bP_N$
as in Definition~\ref{def:bPN} is the
\emph{wetting model} $\bP_N^{\xi,+}$, defined by
\begin{equation*}
	\bP_N^{\xi,+}(\,\cdot\,) := \bP_N^{\xi}(\,\cdot\,|\,s_1 \ge 0 , \, s_2 \ge 0, \,
	\ldots, \, s_N \ge 0 \,) \,.
\end{equation*}
The bulk excursion
law is now $P_t^\bulk = \P_t^\excursion$,
while the final excursion law is
$P_t^\rig = \P_t^\meander$.

\smallskip

Finally, \emph{constrained} versions
of the pinning and wetting models
also fit Definition~\ref{def:bPN}:
\begin{equation*}
	\bP_N^{\xi,c}(\,\cdot\,) := \bP_N^{\xi}(\,\cdot\,|\, s_N = 0) \,, \qquad
	\bP_N^{\xi,+,c}(\,\cdot\,) := \bP_N^{\xi,+}(\,\cdot\,|\, s_N = 0) \,.
\end{equation*}
The final and bulk excursion laws coincide ($P_t^\rig = \P_t^\excursiontwo$
for $\bP_N^{\xi,c}$, $P_t^\rig = \P_t^\excursion$
for $\bP_N^{\xi,+,c}$).

\smallskip

Proposition~\ref{th:cor}
and Remark~\ref{rem:n0} yield immediately the following result.

\smallskip

\begin{theorem}[Tightness for pinning and wetting models]\label{th:ti}
Fix a real sequence $\xi = (\xi_n)_{n\in \N}$.
Under Assumption~\ref{ass:rw}, 
the diffusive rescalings $(\sfP_N \circ \cR_N^{-1})_{N\in\N}$
of pinning or wetting models
 $\sfP_N \in
\{ \bP_N^\xi, \, \bP_N^{\xi,+}, \, \bP_N^{\xi,c} , \, \bP_N^{\xi,+,c}\}$ 
are tight in $C([0,1])$.
\end{theorem}

\smallskip

This result fills a gap in the proof of 
\cite[Lemma~4]{cf:DGZ}, which was also used in the works
\cite{cf:CGZ2}, \cite{cf:CGZ1}.
A recent application of Theorem~\ref{th:ti} can be found in 
\cite{cf:DO}.

\smallskip

Pinning and wetting models are challenging models, which
display a rich behavior.
This complexity is hidden in the regeneration law $p_N = p_N^\xi$.
This explains
the importance of having criteria for tightness, such as Theorem~\ref{th:main1},
which only looks at excursions.

\smallskip

\begin{remark}
There are models where regeneration epochs
are a strict subset of the zero level set.
For instance, in presence of a Laplacian interaction \cite{cf:BC,cf:CD1,cf:CD2}, 
couples of adjacent zeros are regeneration epochs.
Theorem~\ref{th:main1} can cover these cases.
\end{remark}

\section{Proof of Theorem~\ref{th:main2}}
\label{sec:prmain2}

We fix a random walk $(S_n)_{n\in\N_0}$ which satisfies Assumption~\ref{ass:rw},
for simplicity with $\sigma^2 = 1$.
We split the proof of Theorem~\ref{th:main2} in three steps.
To prove \eqref{eq:ui} we may take $n \ge n_0$, with $n_0$ as in
Assumption~\ref{ass:rw}, because \eqref{eq:ui} holds for any fixed $n$,
by $M_n \in L^2$.

\medskip

\noindent
\textbf{Step 1.}
We use the shorthand UI for ``uniformly integrable''.
In this step assume that
\begin{equation}\label{eq:goal1a}
	\tfrac{M_n^2}{n} \text{ under $\P_n$ 
	(resp.\ under $\P^\meander$) is UI} \,, 
\end{equation}
and we show that
\begin{equation}\label{eq:goal1b}
	\tfrac{M_n^2}{n} \text{ under $\P_n^\bridge$
	(resp.\ under $\P_n^\excursion$)  is UI} \,.
\end{equation}

Let us set
$M_{[a,b]} := \max_{a \le i \le b} |S_i|$. Since $M_n \le 
	\max\{M_{[0,n/2]}, M_{[n/2, n]}\}$,
it suffices to prove that $M_{[0,n/2]}^2/n$ and $M_{[n/2, n]}^2/n$
are UI. By symmetry, 
\eqref{eq:goal1b} is equivalent to
\begin{equation}\label{eq:goal1bis}
	\tfrac{M_{n/2}^2}{n} \text{ under $\P_n^\bridge$
	(resp.\ under $\P_n^\excursion$)  is UI} \,.
\end{equation}

We take $n$ even (for simplicity). 
We show that the laws of $\bV_{n/2} := (S_1, \ldots, S_{n/2})$
under $\P_n^\bridge$ (resp.\ $\P_n^\excursion$)
and under $\P_n$ (resp.\ $\P_n^\meander$)
\emph{have a bounded Radon-Nikodym density}:
\begin{equation}\label{eq:RNbound}
	\sup_{n \ge n_0}
	\, \sup_{\bz \in \R^{n/2}} \frac{\P_n^\bridge( \bV_{n/2} \in \dd \bz )}
	{\P_n( \bV_{n/2} \in \dd \bz )}
	< \infty \quad\
	\Bigg( \text{resp.} \  \sup_{n\ge n_0}
	\, \sup_{\bz \in \R^{n/2}} \frac{\P_n^\excursion( \bV_{n/2} \in \dd \bz )}
	{\P_n^\meander( \bV_{n/2} \in \dd \bz )}
	< \infty \Bigg) \,.
\end{equation}
Since $M_{n/2}$ is a function of $\bV_{n/2}$,
it follows that \eqref{eq:goal1a} implies \eqref{eq:goal1bis}
(note that $M_{n/2} \le M_n$).

\smallskip

It remains to prove \eqref{eq:RNbound}. 
By Gnedenko's local limit theorem, in the discrete case
\begin{equation} \label{eq:llteffdisc}
	\forall n \ge n_0: \qquad
	\P(S_n = 0) \ge \tfrac{c}{\sqrt{n}} \,, \qquad
	\sup_{x\in\Z} \, \P(S_n = x) \le \tfrac{C}{\sqrt{n}} \,,
\end{equation}
hence
\begin{equation*} 
	\frac{\P_n^\bridge( \bV_{n/2} = \bz )}{\P_n( \bV_{n/2} = \bz )}
	 = \frac{\P(S_{n/2} = -z_{n/2})}{\P(S_n = 0)}
	\le \frac{C}{c} < \infty \,,
\end{equation*}
which proves the first relation in \eqref{eq:RNbound} in the discrete case.
The continuous case is similar, since
$f_n(0) \ge \frac{c}{\sqrt{n}}$ and
$\sup_{x\in\R} \, f_n(x) \le \frac{C}{\sqrt{n}}$ for $n\ge n_0$,
under Assumption~\ref{ass:rw}.

To prove the second relation in \eqref{eq:RNbound}, in the discrete case we compute 
\begin{equation*}
	\frac{\P_n^\excursion( \bV_{n/2} = \bz )}{\P_n^\meander( \bV_{n/2} = \bz )}
	 = \frac{\P_0(S_1 > 0, \ldots, S_n > 0) \, \P_{z_{n/2}}(S_1 > 0, \ldots,
	 S_{n/2-1} > 0, S_{n/2} = 0)}{\P_0(S_1 > 0, \ldots,
	 S_{n-1} >0 , S_n = 0) \, \P_{z_{n/2}}(S_1 > 0, \ldots,
	 S_{n/2} > 0)} \,,
\end{equation*}
where $\P_{x}$ is the law of the random walk started at $S_0 = x$.
For some $c_1 < \infty$ we have
\begin{equation*}
	\P_0(S_1 > 0, \ldots, S_n > 0) \le \tfrac{c_1}{\sqrt{n}} \,,
\end{equation*}
by \cite[Th.1 in \S XII.7, Th.1 in \S XVIII.5]{cf:Feller2}.
Next we apply \cite[eq. (4.5) in Prop.~4.1]{cf:CarCha2}
(with $a_n = \sqrt{n} \, (1+o(1))$),
which summarizes \cite{cf:AD,cf:VV}:
for some $c_2 \in (0,\infty)$
\begin{equation*}
	\P_0(S_1 > 0, \ldots, S_{n-1} > 0, S_{n} = 0)
	\ge \tfrac{c_2}{n^{3/2}} \,.
\end{equation*}
As a consequence, if we rename $n/2$ as $n$
and $z_{n/2}$ as $x$, it remains to show that
\begin{equation}\label{eq:itri}
	\sup_{n\ge n_0} \,
	\sup_{x \ge 0} \, \frac{n \, \P_x(S_1 > 0, \ldots, S_{n-1} > 0, S_n = 0)}{\P_x(S_1 > 0,
	\ldots, S_n > 0)} < \infty \,.
\end{equation}

By contradiction, if \eqref{eq:itri} does not hold,
there are subsequences $n=n_k \in \N$,
$x=x_k \ge 0$, for $k\in\N$, such that \emph{the ratio in \eqref{eq:itri}
diverges as $k\to\infty$}. We distinguish two cases: either
$\liminf_{k\to\infty} x_k / \sqrt{n_k} = \eta > 0$ (case 1), or
$\liminf_{k\to\infty} x_k / \sqrt{n_k} = 0$, i.e.\
there is a subsequence $k_\ell$ with
$x_{k_\ell} = o(\sqrt{n_{k_\ell}})$ (case 2).

In case 1, i.e. for $x \ge \eta \sqrt{n}$,
the denominator in \eqref{eq:itri} is
bounded away from zero:
\begin{equation*}
	\P_x(S_1 > 0,
	\ldots, S_n > 0)
	\ge \P_{\lfloor \eta \sqrt{n} \rfloor}(S_1 > 0,
	\ldots, S_n > 0)
	\xrightarrow[n\to\infty]{} \P_\eta (B_t > 0 \ \forall t \in [0,1]) > 0 \,,
\end{equation*}
by Donsker's invariance principle \cite{cf:Donsker}
($B = (B_t)_{t\ge 0}$ is Brownian motion started at $\eta$).
For the numerator, 
by \cite[eq. (4.4) in Prop.~4.1]{cf:CarCha2}
which summarizes \cite{cf:Car05,cf:VV},
\begin{gather*}
	\P_{x}(S_1 > 0, \ldots, S_{n-1} > 0, S_{n} = 0)
	\le \tfrac{c_3}{\sqrt{n}} \,
	\P(S_1 < 0, \ldots,
	S_n < 0) \le \tfrac{c_3'}{n} \,,
\end{gather*}
for suitable $c_3, c_3' \in (0,\infty)$. 
Then the ratio in \eqref{eq:itri} is bounded,
which is a contradiction.

In case 2, i.e.\ for  $x = o(\sqrt{n})$,
by \cite[eq. (4.5) in Prop.~4.1]{cf:CarCha2}
we have
\begin{equation} \label{eq:sure}
	\P_{x}(S_1 > 0, \ldots, S_{n-1} > 0, S_{n} = 0)
	\,\underset{n\to\infty}{\sim}\, \underline{V}^-(x) \, \tfrac{c_4}{n^{3/2}} \,,
\end{equation}
for a suitable $\underline{V}^-(x)$.
Since $\{S_1 > 0, \ldots, S_{n} > 0\} 
= \bigcup_{m > n} \{S_1 > 0, \ldots, S_{m-1} > 0, S_{m} = 0\}$ a.s.\ (note that the random walk
is recurrent), we get
\begin{equation*}
	\P_{x}(S_1 > 0, \ldots, S_{n-1} > 0, S_{n} > 0)
	\,\underset{n\to\infty}{\sim}\, \underline{V}^-(x) \, \tfrac{2 \, c_4}{\sqrt{n}} \,,
\end{equation*}
see also \cite[Cor.~3]{cf:Don12}.
Thus the ratio in \eqref{eq:itri} is bounded, which is the desired
contradiction.

This completes the proof of the second relation in \eqref{eq:RNbound} in the discrete case.
The continuous case is dealt with with identical arguments, exploiting
\cite[Th.~5.1]{cf:CarCha2}. \qed

\medskip
\noindent
\textbf{Step 2.}
In this step we assume that
\begin{equation}\label{eq:goal3}
	\tfrac{S_n^2}{n} \text{ under $\P_n$ (resp.\ under $\P_n^\meander$)
	is UI} \,,
\end{equation}
and we deduce that
\begin{equation}\label{eq:goal2}
	\tfrac{M_n^2}{n} \text{ under $\P_n$ (resp.\ under $\P_n^\meander$)
	is UI} \,.
\end{equation}

Observe that $(|S_i|)_{0 \le i \le n}$ is a submartingale
under $\P_n$.
Let us show that $(|S_i|)_{0 \le i \le n}$ is a submartingale also under $\P_n^\meander$
(for every fixed $n\in\N$).
We set for $m\in\N$ and $x \in \R$
\begin{equation*}
	q_m(x) := \P(x + S_1 > 0, x + S_2 > 0, \ldots, x + S_m > 0) \,,
\end{equation*}
with $q_0(x) := 1$. Then
we can write, for any $n \in\N$, $i \in \{0,1,\ldots, n-1\}$ and $x \ge 0$,
\begin{equation*}
	\P_n^\meander \big[ S_{i+1} \in \dd y \,\big|\, S_i = x \big] = 
	\frac{\ind_{(0,\infty)}(y) \, q_{n-(i+1)}(y)}{q_{n-i}(x)} \, \P(X_1 \in \dd y - x) \, .
\end{equation*}
Since $y \mapsto (y-x)$ and $y \mapsto \ind_{(0,\infty)}(y) \, q_{n-(i+1)}(y)$ 
are non-decreasing functions, it follows by the Harris inequality
(a special case of the FKG inequality) and $\E[X_1] = 0$ that
\begin{equation*}
\begin{split}
	\E_n^\meander \big[ S_{i+1} - S_i \,\big|\, S_i = x \big] 
	& \,\ge\, \int_\R (y-x)\, \P(X_1 \in \dd y - x) \cdot
	\int_0^\infty
	\frac{q_{n-(i+1)}(y)}{q_{n-i}(x)} \, \P(X_1 \in \dd y - x)
	 \,=\,  0 \,.
\end{split}
\end{equation*}

Since $(|S_i|)_{0 \le i \le n}$ is a submartingale,
also
$(Z_i := (|S_i| - K)^+)_{0 \le i \le n}$ is a submartingale,
for any $K \in (0,\infty)$.
Doob's $L^2$ inequality yields,
for $\sfP_n = \P_n$ or $\sfP_n = \P_n^\meander$
(recall \eqref{eq:Mn}),
\begin{equation*}
	\sfE_n \big[ (M_n - K)^2 \, \ind_{\{M_n > K\}} \big]
	= \sfE_n\Big[ \Big(\max_{0 \le i \le n} Z_i \Big)^2 \Big]
	\,\le\, 4 \, \sfE_n\big[ Z_n^2 \big]  
	= 4 \, \sfE_n \big[ (S_n - K)^2 \, \ind_{\{S_n > K\}} \big] \,.
\end{equation*}
For $M_n > 2K$ we can bound 
$M_n^2 \le 4(M_n - K)^2$. Since $(S_n - K)^2 \le S_n^2$ for $S_n > K$, we get
\begin{equation*}
	\sfE_n \big[ M_n^2 \, \ind_{\{M_n > 2K\}} \big]
	\le 16 \, \sfE_n \big[ S_n^2 \, \ind_{\{S_n > K\}} \big] \,.
\end{equation*}
We finally choose $K = \frac{1}{2} \sqrt{tn}$, for $t \in (0,\infty)$, to obtain
\begin{equation*}
	\sfE_n \Big[ \tfrac{M_n^2}{n} \, \ind_{\{\frac{M_n^2}{n} > t
	\}} \Big] \le 16 \, 
	\sfE_n \Big[ \tfrac{S_n^2}{n} \, \ind_{\{\frac{S_n^2}{n} >
	\frac{t}{2}\}} \Big] \,, \qquad \forall t > 0 \,.
\end{equation*}
This relation for $\sfP_n = \P_n$ (resp.\ $\sfP_n = \P_n^\meander$)
shows that \eqref{eq:goal3} implies \eqref{eq:goal2}.\qed

\medskip
\noindent
\textbf{Step 3.}
In this step we prove that \eqref{eq:goal3} holds, completing the
proof of Theorem~\ref{th:main2}.
We are going to apply the following standard result, proved below.

\begin{proposition}\label{prop:wc}
Let $(Y_n)_{n\in\N}$, $Y$ be random variables in $L^1$, such that
$Y_n \to Y$ in law. Then $(Y_n)_{n\in\N}$ is UI
if and only if $\lim_{n\to\infty} \E[|Y_n|] = \E[|Y|]$.
\end{proposition}

Let us define
\begin{equation*}
	Y_n := \tfrac{S_n^2}{n} \,.
\end{equation*}
Since $S_n/\sqrt{n}$ under $\P_n$ converges in law
to $Z \sim N(0,1)$,
we have $Y_n \to Z^2$ in law.
Since $\E_n[|Y_n|] = 1 = \E[Z^2]$ for all $n\in\N$,
relation \eqref{eq:goal3} under $\P_n$ follows by Proposition~\ref{prop:wc}.

\smallskip

Next we focus on $\P_n^\meander$.
It is known
\cite{cf:B} that $S_n / \sqrt{n}$ under $\P_n^\meander$ converges in law
toward the Brownian meander at time $1$, 
that is a random variable $V$ with law
$\P(V \in \dd x) := x \, e^{-x^2/2} \, \ind_{(0,\infty)}(x) \, \dd x$.
Therefore $Y_n \to V^2$ in law, under $\P_n^\meander$. 
Since $\E[V^2] = 2$, 
relation \eqref{eq:goal3} under $\P_n^\meander$ is proved once we show that
\begin{equation}\label{eq:fingoa}
	\lim_{n\to\infty} \E_n^\meander \Big[ \tfrac{S_n^2}{n} \Big] = 2 \,.
\end{equation}

To evaluate this limit, we express the law of $S_n/\sqrt{n}$ under $\P_n^\meander$
using fluctuation theory for random walks. By
\cite[equations (3.1) and (2.6)]{cf:Car05}, as $n\to\infty$
\begin{equation*}
	\P_n^\meander\Big( \tfrac{S_n}{\sqrt{n}} \in \dd x \Big) = 
	\big( \sqrt{2\pi}+o(1) \big) \, 
	\int_0^1 \int_0^\infty
	\P\Big( \tfrac{S_{\lfloor n(1-\alpha) \rfloor}}{\sqrt{n}} \in \dd x - \beta \Big) \,
	\ind_{[0,x)}(\beta) \,
	\dd \mu_n(\alpha,\beta)\,,
\end{equation*}
where $\mu_n$ is a finite measure on $[0,1) \times [0,\infty)$, defined
in \cite[eq. (3.2)]{cf:Car05}. Then
\begin{equation*}
\begin{split}
	\E_n^\meander \Big[ \tfrac{S_n^2}{n} \Big] 
	= \big( \sqrt{2\pi}+o(1) \big) 
	\int_0^1 \int_0^\infty
	\Big\{ & \E\Big[ \tfrac{(S_{\lfloor n(1-\alpha) \rfloor}^+)^2}{n} \Big]
	+ 2 \beta \, \E\Big[ \tfrac{S_{\lfloor n(1-\alpha) \rfloor}^+}{\sqrt{n}} \Big]
	+ \beta^2 \, \P
	\Big( \tfrac{S_{\lfloor n(1-\alpha) \rfloor}}{\sqrt{n}} > 0 \Big) 
	\Big\} \, \dd \mu_n \,.
\end{split}
\end{equation*}
By the convergence in law (under $\P$) $S_n/\sqrt{n} \to Z \sim N(0,1)$,
together with the uniform integrability of $(S_n/\sqrt{n})^2$ that we already proved,
we have as $n\to\infty$
\begin{gather*}
	\E\Big[ \tfrac{(S_{\lfloor n(1-\alpha) \rfloor}^+)^2}{n} \Big] 
	\longrightarrow (1-\alpha) \, \E[(Z^+)^2] = \frac{1-\alpha}{2} \,, \\
	\E\Big[ \tfrac{S_{\lfloor n(1-\alpha) \rfloor}^+}{\sqrt{n}} \Big]
	\longrightarrow \sqrt{1-\alpha} \, \E[Z^+] = \tfrac{\sqrt{1-\alpha}}{\sqrt{2\pi}} \,,
	\qquad
	\P\Big( \tfrac{S_{\lfloor n(1-\alpha) \rfloor}}{\sqrt{n}} > 0 \Big) 
	\longrightarrow \P(Z > 0) = \tfrac{1}{2} \,,
\end{gather*}
uniformly for $\alpha \in [0,1-\delta]$,
for $\delta > 0$. 
By \cite[Prop.~5]{cf:Car05} we have the weak convergence
\begin{equation*}
	\mu_n(\dd \alpha, \dd \beta) \Longrightarrow
	\mu(\dd \alpha, \dd \beta) := \frac{\beta}{\sqrt{2\pi} \, \alpha^{3/2}} \, 
	e^{-\frac{\beta^2}{2\alpha}} \, \dd \alpha \, \dd \beta \,,
\end{equation*}
and note that $\mu$ is a finite measure on $[0,1) \times [0,\infty)$.
Then $\lim_{n\to\infty} \E_n^\meander \Big[ \frac{S_n^2}{n} \Big]$ equals
\begin{equation*}
\begin{split}
	\int_0^1 \Big( \int_0^\infty 
	\Big\{ \tfrac{1-\alpha}{2} + 2 \beta \, \tfrac{\sqrt{1-\alpha}}{\sqrt{2\pi}} +
	\tfrac{1}{2}  \beta^2
	\Big\} \, \tfrac{\beta}{\alpha^{3/2}} \, 
	e^{-\frac{\beta^2}{2\alpha}} \, \dd \beta \Big) \dd \alpha
	 = \int_0^1 \Big\{ \tfrac{1-\alpha}{2\sqrt{\alpha}} + \sqrt{1-\alpha}
	+ \sqrt{\alpha} \Big\} \, \dd \alpha = 2 \,,
\end{split}
\end{equation*}
which completes the proof of \eqref{eq:fingoa}.\qed

\begin{proof}[Proof of Proposition~\ref{prop:wc}]
We assume that $Y_n \to Y$ a.s.,
by Skorokhod's representation theorem.
If $(Y_n)_{n\in\N}$ is UI, then $Y_n \to Y$
in $L^1$, hence $\E[|Y_n|] \to \E[|Y|]$.

Assume now that $\lim_{n\to\infty} \E[|Y_n|] = \E[|Y|] < \infty$.
Since $Y_n \to Y$ a.s., dominated convergence yields
$\lim_{n\to\infty} \E[|Y_n| \ind_{\{|Y_n| \le T\}}]
= \E[|Y| \ind_{\{|Y| \le T\}}]$
for $T \in (0,\infty)$ with $\P(|Y|=T) = 0$. Then
\begin{equation*}
	\lim_{n\to\infty} \E[|Y_n| \ind_{\{|Y_n| > T\}}] 
	= \lim_{n\to\infty}
	\big( \E[|Y_n|] - \E[|Y_n| \ind_{\{|Y_n| \le T\}}]  \big)
	= \E[|Y| \, \ind_{\{|Y| > T\}}] \,.
\end{equation*}
Since $\lim_{T\to\infty} \E[|Y| \, \ind_{\{|Y| > T\}}] = 0$,
this shows that $(Y_n)_{n\in\N}$ is UI.
\end{proof}

\section{Proof of Theorem~\ref{th:main3}}
\label{sec:m3}

We fix a random walk $(S_n)_{n\in\N_0}$ which satisfies Assumption~\ref{ass:rw}
in the discrete case (the continuous case
is covered by Theorem~\ref{th:main2}), with $\sigma^2=1$.
We proceed in two steps.

\medskip
\noindent
\textbf{Step 1.}
We assume that $M_n^2/n$ under
$\P_n^\meandertwo$ is UI and we prove that $M_n^2/n$ under
$\P_n^\excursiontwo$ is UI.
As in Section~\ref{sec:prmain2}, it suffices to show that,
with $\bV_{n/2} := (S_1, \ldots, S_{n/2})$ and $n_0$ as in Assumption~\ref{ass:rw},
\begin{equation}\label{eq:RNbound2}
	\sup_{n\ge n_0}
	\, \sup_{\bz \in \Z^{n/2}} \frac{\P_n^\excursiontwo( \bV_{n/2} = \bz )}
	{\P_n^\meandertwo( \bV_{n/2} = \bz )}
	< \infty \,.
\end{equation}

If we define $T := \min\{n\in\N: \ S_n = 0\}$,
we can compute
(recall \eqref{eq:laws2})
\begin{equation*}
	\frac{\P_n^\excursiontwo( \bV_{n/2} = \bz )}{\P_n^\meandertwo( \bV_{n/2} = \bz )}
	 = \frac{\P(T > n) \, \P_{z_{n/2}}(T = n/2)}{\P(T = n) 
	 \, \P_{z_{n/2}}(T > n/2)} \,,
\end{equation*}
where $\P_{x}$ is the law of the random walk started at $S_0 = x$.
By \cite{cf:Kes}, as $n\to\infty$
\begin{equation} \label{eq:localT}
	\P(T = n) = \frac{\sigma}{\sqrt{2\pi} \, n^{3/2}} \, \big(1+o(1)\big) \,,
\end{equation}
hence, summing over $n$, we get $\P(T > n) = 2n \, \P(T=n) \, (1+o(1))$.
Then \eqref{eq:RNbound2} reduces to
\begin{equation}\label{eq:itri2}
	\sup_{n\ge n_0} \,
	\sup_{x \ge 0} \, \frac{n \, \P_x(T=n)}{\P_x(T>n)} < \infty \,.
\end{equation}

Arguing as in the lines after \eqref{eq:itri},
we need to show that the ratio in \eqref{eq:itri2}
is bounded in two cases: when $x \ge \eta \sqrt{n}$ for fixed $\eta > 0$ (\emph{case 1})
and when $x = x_n = o(\sqrt{n})$ (\emph{case 2}).

In case 1, i.e.\ for $x \ge \eta \sqrt{n}$, the denominator in \eqref{eq:itri2}
is bounded away from zero:
\begin{equation*}
	\P_x(T>n) \ge \P_{\lfloor \eta \sqrt{n}\rfloor} (S_1 > 0, \ldots, S_n > 0)
	\xrightarrow[N\to\infty]{} \P_\eta(B_t > 0 \ \forall t \in [0,1]) > 0 \,,
\end{equation*}
where $(B_t)_{t\ge 0}$
is a Brownian motion \cite{cf:Donsker}.
Then the ratio in \eqref{eq:itri2} is bounded because
$\sup_{x\in\Z} \, \P_x(T = n) \le \frac{c'}{n}$
for some $c' \in (0,\infty)$, by \cite[Cor.~1]{cf:K0}.

In case 2, i.e.\ for $x = o(\sqrt{n})$, 
we apply \cite[Thm. 1.1]{Uch11}, which generalizes \eqref{eq:localT}:
\begin{equation*}
	\P_{x}(T=n) = a^*(x) \, \frac{\sigma}{\sqrt{2\pi} \, n^{3/2}}
	\, \big(1+o(1)\big) \qquad
	\text{as } n \to \infty \,, \ \text{uniformly in $x\in\Z$} \,,
\end{equation*}
for a suitable $a^*(x)$ (the potential kernel of the walk). Then
$\P_{x}(T>n) = 2n \, \P_{x}(T=n) \, (1+o(1))$,
hence the ratio
in \eqref{eq:itri2} is bounded.
This completes the proof of \eqref{eq:RNbound2}.\qed

\medskip
\noindent
\textbf{Step 2.}
We prove that $M_n^2/n$ under
$\P_n^\meandertwo$ is UI.
We argue by contradiction: if this does not hold, then there are $\eta > 0$
and $(n_i)_{i\in\N}$, $(K_i)_{i\in\N}$,
with $\lim_{i\to\infty} K_i = \infty$, such that
\begin{equation}\label{eq:contra}
	\E_{n_i}^\meandertwo\Big[ \tfrac{M_{n_i}^2}{n_i} \, 
	\ind_{\{\frac{M_{n_i}^2}{n_i}
	> K_i\}} \Big] \ge \eta \,, \qquad \forall i \in \N \,.
\end{equation}
We are going to deduce that $M_n^2/n$ under $\P_n$ is not UI,
which contradicts Theorem~\ref{th:main2}.

\smallskip

We show below that we can strengthen \eqref{eq:contra},
replacing $\E_{n_i}^\meandertwo$ by $\E_{m}^\meandertwo$ for any $m \in \{n_i,\ldots,
2n_i\}$: more precisely, there exists $\eta' > 0$ such that
\begin{equation}\label{eq:contra2}
	\E_{m}^\meandertwo\Big[ \tfrac{M_{n_i}^2}{n_i} \, 
	\ind_{\{\frac{M_{n_i}^2}{n_i}
	> K_i\}} \Big] \ge \eta' \,, \qquad
	\forall i \in \N \,, \
	\forall m \in \{n_i, \ldots, 2n_i\} \,.
\end{equation}
To exploit \eqref{eq:contra2},
we work on the time horizon $2n$, for fixed $n\in\N$.
We split any path $S = (S_0, \ldots, S_{2n})$ with $S_0 = 0$ in two parts
$\tilde S = (S_0, S_1, \ldots, S_\sigma)$ and
$\hat S = (S_{\sigma}, S_{\sigma + 1}, \ldots, S_{2n})$,
where $\sigma := \sigma_{2n} := \max\{i \in \{0,\ldots, 2n\}: \ S_i = 0\}$.
If $S$ is chosen according to the unconditioned law $\P_{2n}$,
then $\hat S$ has law $\P_{2n-\sigma}^\meandertwo$, conditionally on $\sigma$.
If we set $\hat M_{2n} := \max |\hat S| = \max_{\sigma \le i \le 2n} |S_i|$,
the bound $M_{2n} \ge \hat M_{2n}$ gives
\begin{equation*}
\begin{split}
	\E \Big[ \tfrac{(M_{2 n})^2}{2n} \, 
	\ind_{\{\frac{(M_{2 n})^2}{2n}
	> \frac{K}{2}\}} \Big] 
	& \ge
	\E \Big[ \tfrac{(\hat M_{2 n})^2}{2n} \, 
	\ind_{\{\frac{(\hat M_{2 n})^2}{2n}
	> \frac{K}{2}\}} \Big] 
	= \sum_{r=0}^{2n} \E \bigg[ 
	\E_{2n-r}^\meandertwo \Big[ \tfrac{M_{2 n - r}^2}{2n} \, 
	\ind_{\{\frac{M_{2 n - r}^2}{n}
	> K\}} \Big] \, \ind_{\{\sigma=r\}} \bigg] \,.
\end{split}
\end{equation*}
We now restrict the sum to $r \le n$, so that
$M_{2n - r}^2 \ge M_{n}^2$, to get
\begin{equation*}
\begin{split}
	\E \Big[ \tfrac{(M_{2 n})^2}{2n} \, 
	\ind_{\{\frac{(M_{2 n})^2}{2n}
	> \frac{K}{2} \}} \Big] 
	& \ge 
	\tfrac{1}{2} \, \P(\sigma_{2n} \le n)  \, 
	\inf_{n \le m \le 2n} \,
	\E_{m}^\meandertwo \Big[ \tfrac{M_{n}^2}{n} \, 
	\ind_{\{\frac{M_{n}^2}{n}
	> K\}} \Big]  \,.
\end{split}
\end{equation*}
Note that $\lim_{n\to\infty} \P(\sigma_{2n} \le n)
= \P( B_t \ne 0 \ \forall t \in (\frac{1}{2}, 1]) =: p > 0$ (actually $p=\frac{1}{2}$, by the arcsine law),
hence $\gamma := \inf_{n\in\N} \P(\sigma_{2n} \le n) > 0$.
If we take $n = 2 n_i$ and $K = K_i$, by \eqref{eq:contra2}
\begin{equation*}
	\liminf_{K\to\infty} \ \sup_{n\in\N} \
	\E \Big[ \tfrac{(M_{n})^2}{n} \, 
	\ind_{\{\frac{(M_{n})^2}{n}
	> \frac{K}{2} \}} \Big] \,\ge\,
	\inf_{i\in\N} \,
	\E \Big[ \tfrac{(M_{2 n_i})^2}{2n_i} \, 
	\ind_{\{\frac{(M_{2 n_i})^2}{2n_i}
	> \frac{K_i}{2} \}} \Big] 
	\,\ge\, \frac{\gamma \, \eta'}{2} > 0 \,.
\end{equation*}
This means that $M_n^2/n$ under $\P_n$ is not UI,
which contradicts Theorem~\ref{th:main2}.

\smallskip

It remains to prove \eqref{eq:contra2}.
We fix $\sfC \in (0,\infty)$, to be determined later.
We may assume that $K_i \ge \sfC$ for all $i\in\N$.
To deduce \eqref{eq:contra2} from \eqref{eq:contra},
we show that for some $c > 0$
\begin{equation} \label{eq:pargo}
	\inf_{n \in \N, \ m \in \{n, \ldots, 2n\},\ z \in \Z: \, z \ge \sfC \sqrt{n}}
	\ \frac{\P_m^\meandertwo(M_n = z)}{\P_n^\meandertwo(M_n = z)}
	\,\ge\, c  \,.
\end{equation}
Fix $m \ge n$ and $z > 0$. If  we sum over the last $\ell \le n$ 
for which $M_n = |S_\ell|$, we can write
\begin{equation*}
	\P_m^\meandertwo(M_n = z)
	= \sum_{\ell = 1}^n \P_m^\meandertwo( M_{\ell - 1} \le z, \, |S_\ell| = z, \,
	|S_i| < z \ \forall i= \ell+1, \ldots, n) \,.
\end{equation*}
We write $\P_m^\meandertwo(\,\cdot\,) = \P(\,\cdot\,|\, E_m)$, with
$E_m := \{S_1 \ne 0, \ldots, S_m \ne 0\}$, and we apply the Markov property at time $\ell$.
The cases $S_\ell = z$ and $S_\ell = -z$ give a similar contribution and
we do not distinguish between them
(e.g.\ assume that the walk is symmetric). Then
\begin{equation*}
\begin{split}
	\P_m^\meandertwo(M_n = z)
	= \frac{1}{\P(T > m)} \sum_{\ell = 1}^n
	& \ \P( M_{\ell - 1} \le z, \, |S_\ell| = z, \, E_\ell) \,  
	\underset{A}{\underbrace{
	\P_z ( |S_i| < z \ \forall 1 \le i \le n-\ell , \ E_{m-\ell})}} \,.
\end{split}
\end{equation*}
The same expression holds if we replace
$\P_m^\meandertwo$ by $\P_n^\meandertwo$, namely
\begin{equation*}
\begin{split}
	\P_n^\meandertwo(M_n = z)
	= \frac{1}{\P(T > n)} \sum_{\ell = 1}^n
	& \ \P( M_{\ell - 1} \le z, \, |S_\ell| = z, \, E_\ell) \, 
	\underset{B}{\underbrace{
	\P_z ( |S_i| < z \ \forall 1 \le i \le n-\ell , \ E_{n-\ell})}} \,.
\end{split}
\end{equation*}
Since $\P(T > m) \le \P(T > n)$, to prove \eqref{eq:pargo} 
we show that $A \ge c \, B$, with $c>0$. We bound
\begin{equation*}
	B \le \P_z ( S_i < z \ \forall i = 1 , \ldots,  n-\ell)
	= \P_0( E_{n-\ell}^-) \,,
\end{equation*}
where we set $E_{k}^- := \{S_1 < 0, \ldots, S_k < 0\}$.
Similarly, for $z \ge \sfC\sqrt{n}$  we bound
\begin{equation*}
\begin{split}
	A & \ge 
	\P_z ( S_i < z \ \forall i=1, \ldots, n-\ell , 
	\ S_i > 0 \ \forall i=1, \ldots, m-\ell) \\
	& = \P_0 ( E_{n-\ell}^- , 
	\ S_i > -z \ \forall i=1, \ldots, m-\ell)  \ge \P_0 ( E_{n-\ell}^- ) \,
	\underset{D}{\underbrace{\P_0 \big( (-S_i) < \sfC \sqrt{n} \ \forall i=1, \ldots, m-\ell
	\,\big|\, E_{n-\ell}^- \big)}} \,.
\end{split}
\end{equation*}
It remains to show that $D \ge c$.
Let us set $\tilde S_i := - S_i$ and $\tilde E_k^+ := E_k^- = \{\tilde S_1 > 0, 
\ldots, \tilde S_k > 0\}$.
If we write $r := n-\ell$,
for $m \in \{n, \ldots, 2n\}$,
we have $m - \ell = r + (m-n) \le r + n$, hence 
\begin{equation} \label{eq:bopr}
	D \ge \P\big( \tilde S_i < \tfrac{1}{2}\sfC \sqrt{r} \ \forall i=1,\ldots, r \,
	\big| \, \tilde E_r^+ \big) 
	\cdot \P\big( \tilde S_i < \tfrac{1}{2}\sfC \sqrt{n} \ \forall i=1,\ldots, n \,\big) \,,
\end{equation}
by the Markov property,
since $(\tilde S_{j})_{j \ge r}$
under $\P(\,\cdot\,| \tilde E_r^+)$ is the random walk $\tilde S$
started at $\tilde S_r$.

By \cite{cf:B,cf:Donsker},
as $r \to \infty$ the two probabilities in the right
hand side of \eqref{eq:bopr} converge respectively to
$\P( \sup_{t\in [0,1]} m_t < \frac{1}{2}\sfC)$
and $\P( \sup_{t\in [0,1]} B_t < \frac{1}{2}\sfC)$, 
where $B = (B_t)_{t\ge 0}$ is Brownian motion and
$m = (m_t)_{t\in [0,1]}$ is Brownian
meander.
Then, if we fix $\sfC > 0$ large enough, the right hand side
of \eqref{eq:bopr} is $\ge c > 0$ for all $r, n \in \N_0$.\qed

\section{Proof of Theorem~\ref{th:main1}}
\label{sec:main1}

Let us set $\bQ_N := \bP_N \circ \cR_N^{-1}$.
We prove that conditions \eqref{it:1} and \eqref{it:2} in Theorem~\ref{th:main1}
are necessary and sufficient for the tightness of $(\bQ_N)_{N\in\N}$.

\medskip

\noindent
\textbf{Necessity.}
The necessity of condition \eqref{it:1} is clear: just note that,
by Definition~\ref{def:bPN},
the law $\bP_N$ coincides with $P_N^\rig$ (resp.\ with $P_N^\bulk$)
if we choose the regeneration
law $p_N$ to be concentrated on the single set $\{0\}$ (resp.\
on the single set $\{0,N\}$).

\smallskip

To prove necessity of condition \eqref{it:2},
we assume by contradiction that \eqref{it:2} fails.
Then there exists $\eta > 0$ and
two sequences $(t_n)_{n\in\N}$, $(a_n)_{n\in\N}$,
with $\lim_{n\to\infty} a_n = \infty$, such that
\begin{equation} \label{eq:sth}
	P_{t_n}^\bulk\bigg( \frac{M_{t_n}}{\sqrt{t_n}}
	> a_n \bigg) \ge \frac{\eta}{a_n^2} \,, \qquad \forall n \in \N \,, \qquad
	\text{where} \quad
	M_t := \max_{0 \le i \le t} |x_i| \,.
\end{equation}
We may assume that $a_n \in \N$ 
(otherwise consider $\lfloor a_n \rfloor$ and redefine $\eta$).

Define $N_n := t_n \, a_n^2$ and let $p_{N_n}$ be the regeneration
law concentrated on the single set 
$\{0, t_n, 2t_n, \ldots, N_n  - t_n, N_n\}$.  Let $\bP_{N_n}$ be the corresponding probability
on $\R^{[N_n]}$, see Definition~\ref{def:bPN}.
We now show that $\bQ_N = \bP_N \circ \cR_N^{-1}$  \emph{is not tight on $C([0,1])$}.

Any path $f(t)$ under $\bQ_{N_n}$ vanishes for 
$t \in \{0, \frac{1}{a_n^2}, \frac{2}{a_n^2}, \ldots,
1 - \frac{1}{a_n^2}, 1\}$, which becomes dense in $[0,1]$ as $n\to\infty$.
Then, if $(\bQ_{N_n})_{n\in\N}$ were tight,
it would converge weakly to the law concentrated
on the single path $(f_t \equiv 0)_{t\in [0,1]}$.
We rule this out by showing that
\begin{equation} \label{eq:cotr}
	\liminf_{n\to\infty} \, \bQ_{N_n} \bigg(
	\sup_{t\in [0,1]} |f_t| > 1 \bigg)
	\ge 1 - e^{-\eta} > 0 \,.
\end{equation}
For $x \in \R^{[N_n]}$ and $j = 1,\ldots, a_n^2$ we define
$M_{t_n}^{(j)} := \max_{i \in \{(j-1) t_n, \ldots, j t_n\}} |x_i|$,
so that 
\begin{equation*}
	\bQ_{N_n} \bigg(
	\sup_{t\in [0,1]} |f_t| > 1 \bigg)
	= \bP_{N_n} \bigg(
	\max_{i=0,1,\ldots, N_n} |x_i| > \sqrt{N_n} \bigg)
	= \bP_{N_n}\Bigg( \max_{j=1, \ldots, a_n^2} 
	\frac{M_{t_n}^{(j)}}{\sqrt{t_n}} > a_n \Bigg) \,.
\end{equation*}
The random variables $M_{t_n}^{(j)}$ for $j=1,\ldots, a_n^2$ are independent
and identically distributed,
because they refer to different excursions. Then we conclude by \eqref{eq:sth}:
\begin{equation*}
\begin{split}
	\bQ_{N_n} \bigg(
	\sup_{t\in [0,1]} |f_t| > 1 \bigg)
	& = 1 - \Bigg( 1 - 
	P_{t_n}^\bulk\bigg( \frac{M_{t_n}}{\sqrt{t_n}}
	> a_n \bigg) \Bigg)^{a_n^2} \ge 1 - \Bigg( 1 - 
	\frac{\eta}{a_n^2} \Bigg)^{a_n^2}
	\xrightarrow[\ n \to \infty \ ]{} \ 1 - e^{-\eta} \,.
\end{split}
\end{equation*}

\medskip

\noindent
\textbf{Sufficiency.}
We assume that conditions \eqref{it:1} and \eqref{it:2} in Theorem~\ref{th:main1} hold
and we prove that $(\bQ_N)_{N\in\N}$ is tight in $C([0,1])$, that is
\begin{equation} \label{eq:figo}
	\forall \eta > 0: \qquad
	\lim_{\delta \downarrow 0} \ \sup_{N\in\N} \
	\bQ_N\big( \Gamma(\delta) > \eta \big) = 0 \,,
\end{equation}
where $\Gamma(\delta)(f) := \sup_{|t-s| \le \delta} |f_t - f_s|$ 
denotes the continuity modulus of
$f \in C([0,1])$.

Given a finite subset $U = \{u_1 < \ldots < u_n \} \subseteq [0,1]$
and points $s, t \in [0,1]$, we write
$s \sim_U t$ iff 
no point $u_i \in U$ lies
between $s$ and $t$. Then we define
\begin{equation*}
	\tilde \Gamma_{U}(\delta)(f) := 
	\sup_{s,t \in [0,1]: \ s \sim_U t, \ |t-s| \le \delta} |f_t - f_s| \,.
\end{equation*}
Plainly, if $f(u_i) = 0$ for all $u_i \in U$, then
$\Gamma(\delta)(f) \le 2 \, \tilde \Gamma_{U}(\delta)(f)$.
This means that in \eqref{eq:figo} we can replace
$\Gamma(\delta)(f)$ by $\tilde \Gamma_U(\delta)(f)$, where $U$
is any subset of $[0,1]$ on which $f$ vanishes.
We fix $U = \{\frac{t_1}{N}, \ldots, \frac{t_n}{N}\}$,
where $t_i$ are the regeneration epochs of $\bP_N$. It remains to show that
\begin{equation} \label{eq:figo2}
	\forall \eta > 0: \qquad
	\lim_{\delta \downarrow 0} \ \sup_{N\in\N} \
	\bQ_N\big( \tilde\Gamma_U(\delta) > \eta \big) = 0 \,.
\end{equation}

We set for short $Q_t^\rig := P_t^\rig \circ \cR_t^{-1}$ and 
$Q_t^\bulk := P_t^\bulk \circ \cR_t^{-1}$.
By Definition~\ref{def:bPN}
\begin{equation*}
\begin{split}
	\bQ_N\big( \tilde\Gamma_U(\delta) \le \eta \big) =
	\sum_{n=1}^{N+1}  \sum_{0 = t_1 < \ldots < t_n \le N} \,
	\!\!\!\!p_N(\{t_1,\ldots, t_n\}) \,
	& \prod_{i=1}^{n-1} 
	Q_{t_{i+1}-t_i}^\bulk\Big( \Gamma(\tfrac{N}{t_{i+1}-t_i} \delta) \le
	\eta \sqrt{\tfrac{N}{t_{i+1}-t_i}}
	\,\Big) \, \times \\
	& \qquad \times \, Q_{N - t_n}^\rig\Big( \Gamma(\tfrac{N}{N-t_n} \delta) \le
	\eta \sqrt{\tfrac{N}{N-t_n}} \, \Big) \,.
\end{split}
\end{equation*}
Note that we have the original continuity modulus $\Gamma$.
Let us set
\begin{align}
	\label{eq:gbulk}
	g_\eta^\bulk(\delta) &:= \inf_{\substack{N\in\N, \ 2 \le n \le N+1, \\ 
	0 = t_1 < \ldots < t_n \le N}}\,
	\prod_{i=1}^{n-1} 
	Q_{t_{i+1}-t_i}^\bulk\Big( \Gamma(\tfrac{N}{t_{i+1}-t_i} \delta) \le
	\eta \sqrt{\tfrac{N}{t_{i+1}-t_i}}
	\,\Big) \\
	\notag
	g_\eta^\rig(\delta) &:= \inf_{N\in\N, \ 1 \le t < N}\,
	Q_{N-t}^\rig\Big( \Gamma(\tfrac{N}{N-t} \delta) \le
	\eta \sqrt{\tfrac{N}{N-t}} \, \Big) \,,
\end{align}
so that we can bound $\bQ_N\big( \tilde\Gamma_U(\delta) \le \eta \big) 
\ge \, g_\eta^\bulk(\delta) \, g_\eta^\rig(\delta)$
(we recall that $p_N(\cdot)$ is a probability). We complete the proof
of \eqref{eq:figo2} by showing that
\begin{equation*}
	\forall \eta > 0: \qquad
	\lim_{\delta\downarrow 0}
	\, g_\eta^\bulk(\delta) \, g_\eta^\rig(\delta) \ge 1 \,.
\end{equation*}

\smallskip

We first show that $\lim_{\delta \downarrow 0} g_\eta^\rig(\delta) \ge 1$, 
for every $\eta > 0$.
We fix $\theta \in (0,1)$ and consider two regimes.
For $t < (1-\theta) N$ we can bound (recall that
$(Q_{\ell}^\rig)_{\ell\in\N}$ is tight by assumption)
\begin{equation*}
	\inf_{N\in\N, \ 1 \le t < (1-\theta)N}\,
	Q_{N-t}^\rig\Big( \Gamma(\tfrac{N}{N-t} \delta) \le
	\eta \sqrt{\tfrac{N}{N-t}} \, \Big) \,\ge\,
	\inf_{\ell\in\N} \, Q_{\ell}^\rig\big( \Gamma(\tfrac{\delta}{\theta}) \le
	\eta \, \big) \,\xrightarrow[\delta\downarrow 0]{}\, 1 \,.
\end{equation*}
On the other hand, for $t \ge (1-\theta) N$ we can bound
\begin{equation*}
	\inf_{N\in\N, \ (1-\theta)N \le t < N}\,
	Q_{N-t}^\rig\Big( \Gamma(\tfrac{N}{N-t} \delta) \le
	\eta \sqrt{\tfrac{N}{N-t}} \, \Big)
	\,\ge\, \inf_{\ell\in\N} \, Q_{\ell}^\rig
	\Big( \max_{s \in [0,1]}|f_s| \le
	\tfrac{1}{2} \tfrac{\eta}{\sqrt{\theta}} \Big) =: h_\eta(\theta) \,.
\end{equation*}
For any $\eta > 0$, we have $\lim_{\delta\downarrow 0}
	\, g_\eta^\rig(\delta) \,\ge\,
	\lim_{\theta\downarrow 0} \, h_\eta(\theta) = 1$,
by the tightness of $(Q_{\ell}^\rig)_{\ell\in\N}$.

\smallskip

To complete the proof,
we show that $\lim_{\delta \downarrow 0} g_\eta^\bulk(\delta) \ge 1$, for every $\eta > 0$.
Note that
\begin{equation*}
	\inf_{t\in\N} \, Q_{t}^\bulk\Big( \max_{s \in [0,1]} |f_s| \le
	a \,\Big)
	\,=\, \inf_{t\in\N} \, P_{t}^\bulk\Big( \max_{i = 0,\ldots, t} |x_i| \le
	a \, \sqrt{t} \,\Big) \,\ge\, 1 - \frac{\epsilon(a)}{a^2} \,, 
\end{equation*}
where $\lim_{a \uparrow \infty} \epsilon(a) = 0$,
by assumption \eqref{it:2}.
We may assume that $a \mapsto \epsilon(a)$ is non increasing.
Fix $\theta \in (0,1)$.
Given a family of epochs $0 \le t_1 < \ldots < t_n \le N$, we distinguish
two cases.
\begin{itemize}
\item For $\theta N < t_{i+1}-t_i \le N$ we can bound
\begin{equation*}
	Q_{t_{i+1}-t_i}^\bulk\Big( \Gamma(\tfrac{N}{t_{i+1}-t_i} \delta) \le
	\eta \sqrt{\tfrac{N}{t_{i+1}-t_i}}
	\,\Big) \,\ge\,
	\inf_{t\in\N} \, Q_{t}^\bulk\Big( \Gamma(\tfrac{\delta}{\theta}) \le
	\eta \,\Big) \,=:\, F_{\eta,\theta}(\delta) \,,
\end{equation*}
and note that for fixed $\eta,\theta$ we have
$\lim_{\delta \downarrow 0} F_{\eta,\theta}(\delta) = 1$, 
because $(Q_{t}^\bulk)_{t\in\N}$ is tight.

\item For $t_{i+1}-t_i \le \theta N$ we can bound
\begin{equation*}
\begin{split}
	Q_{t_{i+1}-t_i}^\bulk\Big( \Gamma(\tfrac{N}{t_{i+1}-t_i} \delta) \le
	\eta \sqrt{\tfrac{N}{t_{i+1}-t_i}}
	\,\Big) 
	& \,\ge\,
	Q_{t_{i+1}-t_i}^\bulk\Big( \max_{s\in [0,1]} |f_s| \le
	\tfrac{\eta}{2} \sqrt{\tfrac{N}{t_{i+1}-t_i}}
	\,\Big) \\
	& \,\ge\, 1 - \tfrac{4 (t_{i+1}-t_i)}{\eta^2\, N} \, 
	\epsilon\big(\tfrac{\eta}{2 \, \sqrt{\theta}}\big)
	 \,\ge\, \exp\Big( - \tfrac{8 (t_{i+1}-t_i)}{\eta^2\, N} \, 
	\epsilon\big(\tfrac{\eta}{2 \, \sqrt{\theta}}\big) \Big) \,,
\end{split}
\end{equation*}
where the last inequality holds for $\theta > 0$ small,
by $1-z \ge e^{-2z}$ for $z \in [0,\frac{1}{2}]$.
\end{itemize}
We can have $t_{i+1}-t_i > \theta N$ for at most $\lfloor 1/\theta \rfloor$ values of $i$,
hence
\begin{equation*}
\begin{split}
	g_\eta^\bulk(\delta)
	& \,\ge\,  F_{\eta,\theta}(\delta)^{\frac{1}{\theta}} \,
	\prod_{i = 1}^{n-1}
	\exp\Big( - \tfrac{8 (t_{i+1}-t_i)}{\eta^2\, N} \, 
	\epsilon\big(\tfrac{\eta}{2 \, \sqrt{\theta}}\big) \Big)
	\,\ge\, F_{\eta,\theta}(\delta)^{\frac{1}{\theta}} \,
	\exp\Big( - \tfrac{8 }{\eta^2} \, 
	\epsilon\big(\tfrac{\eta}{2 \, \sqrt{\theta}}\big) \Big) \,.
\end{split}
\end{equation*}
Given $\eta > 0$ and $\epsilon > 0$, we first fix $\theta > 0$ small enough,
so that the exponential is greater than $1-\epsilon$;
then we let $\delta \to 0$, so that 
$F_{\eta,\theta}(\delta)^{\frac{1}{\theta}} \to 1$.
This yields $\lim_{\delta \downarrow 0} g_\eta^\bulk(\delta) \ge 1-\epsilon$.
As $\epsilon > 0$ was arbitrary, we get
$\lim_{\delta \downarrow 0} g_\eta^\bulk(\delta) \ge 1$.\qed

\section*{Acknowledgements}

I thank Tal Orenshtein for reviving the interest on the 
problem and for stimulating email exchanges.
I also thank Denis Denisov, Ron Doney, Giambattista Giacomin,
Vitali Wachtel and Lorenzo Zambotti
for discussions and references.
This work is supported by the PRIN Grant 20155PAWZB ``Large Scale Random Structures''.


\smallskip

\begin{thebibliography}{CGZ07a}

\bibitem[AD99]{cf:AD}
L. Alili and R. A. Doney,
\textit{Wiener-Hopf factorization revisited and some applications},
Stoc. Stoc. Rep. {\bf 66} (1999), 87--102.

\bibitem[Bel72]{cf:Bel} B. Belkin,
\textit{An invariance principle for conditioned 
recurrent random walk attracted to a stable law},
Z. Wahr. V. Gebiete {\bf 21} (1972), 45--64.

\bibitem[Bol76]{cf:B} E. Bolthausen,
\textit{On a functional central limit theorem for random walks conditioned to stay positive},
Ann. Probability  {\bf 4}  (1976), 480--485.

\bibitem[BC10]{cf:BC}
M. Borecki and F. Caravenna,
\textit{Localization for (1+1)-dimensional pinning models with
$\nabla+\Delta$ interaction},
Electron. Commun. Probab. {\bf 15} (2010), 534--548.

\bibitem[Car05]{cf:Car05}
F. Caravenna,
\textit{A local limit theorem for random walks conditioned to stay positive},
Probab. Theory Related Fields \textbf{133} (2005), 508--530.

\bibitem[CC13]{cf:CarCha2}
F. Caravenna and L. Chaumont,
\textit{An invariance principle for random walk bridges conditioned to stay positive},
Electron. J. Probab. {\bf 18} (2013), no. 60, 1--32.

\bibitem[CD08]{cf:CD1}
F. Caravenna and J.-D. Deuschel,
\textit{Pinning and wetting transition for (1+1)-dimensional fields 
with Laplacian interaction},
Ann. Probab. {\bf 36} (2008), 2388--2433.

\bibitem[CD09]{cf:CD2}
F. Caravenna and J.-D. Deuschel,
\textit{Scaling limits of (1+1)-dimensional pinning models with Laplacian interaction},
Ann. Probab. {\bf 37} (2009), 903--945.

\bibitem[CGZ06]{cf:CGZ2} F. Caravenna, G. Giacomin and L. Zambotti,
\textit{Sharp asymptotic behavior for wetting models in (1+1)--dimension},
Elect. J. Probab. {\bf 11} (2006), 345--362.

\bibitem[CGZ07a]{cf:CGZ1} F. Caravenna, G. Giacomin and L. Zambotti,
{\it A renewal theory approach to periodic copolymers with
adsorption}, Ann. Appl. Probab. {\bf 17} (2007), 1362-1398.

\bibitem[CGZ07b]{cf:CGZ?}
F. Caravenna, G. Giacomin and L. Zambotti,
\emph{Tightness conditions for polymer measures},
unpublished manuscript (2007),
arXiv.org: math/0702331

\bibitem[DGZ05]{cf:DGZ}
J.--D. Deuschel, G. Giacomin and L. Zambotti,
\textit{Scaling limits of  equilibrium wetting models
in  (1+1)--dimension}, Probab. Theory Rel. Fields {\bf 119} (2005), 471--500.

\bibitem[DO18]{cf:DO}
J.--D. Deuschel and T. Orenshtein,
\textit{Scaling limit of wetting models in 1+1 dimensions pinned to a shrinking strip},
preprint (2018), arXiv.org: 1804.02248 [math.PR].

\bibitem[Don12]{cf:Don12}
R. A. Doney,
\textit{Local behaviour of first passage probabilities},
Probab. Theory Related Fields {\bf 152} (2012), 559--588.

\bibitem[Don51]{cf:Donsker} M. D. Donsker,
\textit{An invariance principle for certain probability limit theorems},
Mem. Amer. Math. Soc. \textbf{6} (1951), 12 pp.

\bibitem[Fel71]{cf:Feller2} W.~Feller,
\textit{An introduction to probability theory and its applications},
Vol. II, Second edition, John Wiley \& Sons (1971).

\bibitem[Gia07]{cf:Gia} G.~Giacomin, \textit{Random polymer models},
Imperial College Press, World Scientific (2007).

\bibitem[Gia11]{cf:Gia2} 
G.~Giacomin,
\textit{Disorder and Critical Phenomena Through Basic Probability Models},
\'Ecole d’\'Et\'e de Probabilit\'es de Saint-Flour XL,
Springer (2011).

\bibitem[Hol09]{cf:dH}
F.\ den Hollander,
\textit{Random Polymers}, 
\'Ecole d'\'Et\'e de Probabilit\'es de Saint-Flour XXXVII,
Springer (2009).

\bibitem[Kai75]{cf:K0} W.D. Kaigh,
\textit{A conditioned local limit theorem for recurrent random walk},
Ann. Probability {\bf 3} (1975), 883--888.

\bibitem[Kai76]{cf:K} W.D. Kaigh,
\textit{An invariance principle for random walk conditioned by a late return to zero},
Ann. Probability {\bf 4} (1976), 115--121.

\bibitem[Kes63]{cf:Kes} H. Kesten,
\textit{Ratio theorems for random walk II},
J. Analyse Math. {\bf 9} (1963), 323--379.

\bibitem[Lig68]{cf:Lig} T.~L.~Liggett, \textit{An invariance principle
for conditioned sums of independent random variables},
J. Math. Mech. {\bf 18} (1968), 559--570.

\bibitem[Uch11]{Uch11}
K. Uchiyama,
\textit{The First Hitting Time of A Single Point for Random Walks},
Electron. J. Probab. 16 (2011), paper n. 71, 1960--2000.

\bibitem[VV09]{cf:VV}
V. A. Vatutin and V. Wachtel,
\textit{Local probabilities for random walks conditioned to stay positive},
Probab. Theory Relat. Fields {\bf 143} (2009) 177--217.

\end{thebibliography}
\end{document}